\documentclass{amsart}
\usepackage[utf8]{inputenc}
\usepackage{amsmath}
\usepackage{amsthm}
\usepackage{amsfonts}
\usepackage{amssymb, amsbsy}
\usepackage{hyperref}
\usepackage{enumitem}
\usepackage{comment}
\usepackage{mathtools}
\usepackage{cancel}
\usepackage{xcolor}
\usepackage{stmaryrd}
\usepackage{bbm}

\DeclareMathAlphabet{\dutchcal}{U}{dutchcal}{m}{n}
\SetMathAlphabet{\dutchcal}{bold}{U}{dutchcal}{b}{n}
\DeclareMathAlphabet{\dutchbcal} {U}{dutchcal}{b}{n}

\title{Characterization of polyconvex isotropic functions}
\author{David Wiedemann}
\address{Institute of Mathematics, University of Augsburg, Universit\"atsstr.~12a, 86159 Augsburg, Germany.
	\newline
	Current address: Department of mathematics, Technical University of Dortmund, Vogelspothsweg 87, 44227
	Dortmund, Germany}
\email{david.wiedemann@tu-dortmund.de}

\author{Malte A. Peter}
\address{Institute of Mathematics \& Centre for Advanced Analytics and Predictive Sciences (CAAPS), University of Augsburg, Universit\"atsstr.~12a, 86159 Augsburg, Germany}
\email{malte.peter@math.uni-augsburg.de}

\subjclass[2020]{49J10, 74B20, 74G65}
\keywords{Polyconvexity, Isotropy, Finite isotropic elasticity}

\newcommand{\R}{\mathbb{R}}
\newcommand{\N}{\mathbb{N}}

\newcommand{\M}{\mathcal{M}}
\newcommand{\m}{\dutchcal{m}}
\newcommand{\diag}{\operatorname{diag}}
\newcommand{\adj}{\operatorname{adj}}
\newcommand{\SvPc}{\operatorname{SvPc}}

\newcommand{\SO}{\mathrm{SO}}
\renewcommand{\O}{\mathrm{O}}
\renewcommand{\P}{\operatorname{P}}
\newcommand{\Pd}{\Pi(d)}
\renewcommand{\S}{\mathrm{S}}

\newtheorem{theorem}{Theorem}[section]
\newtheorem{corollary}[theorem]{Corollary}
\newtheorem{lemma}[theorem]{Lemma}
\newtheorem{proposition}[theorem]{Proposition}
\newtheorem{definition}[theorem]{Definition}
\newtheorem{remark}[theorem]{Remark}
\newcommand{\1}{\mathbbm{1}}

\begin{document}
	\sloppy

	\begin{abstract}
		
		Polyconvexity is an important concept in the analysis of energies related to elasticity. A function $f \colon \R^{d\times d} \to \R$ is called polyconvex if it can be written as a convex function in the minors of the argument.
		We show that for isotropic functions it suffices to consider diagonal matrices. For $d=3$, this leads to a dimension reduction for the convex representative of $f$ from $\R^{19}$ to $\R^7$.
		
		Moreover, we present a new result for the polyconvexity of functions formulated in the principal invariant of the left or right stretch tensor.
	\end{abstract}
	
	\maketitle
	
	\section{Introduction}\label{sec:Intro}
	
	Many applications in elasticity aim to find global minimizers of functionals of the form
	\begin{equation*}
		I \colon W^{1,p}(\Omega) \to \R_\infty \coloneqq \R \cup \{ \infty \}; \quad I(u) \coloneqq \int\limits_\Omega W(Du(x)) \, \textrm{d}x
	\end{equation*}
	for $\Omega \subset \R^d$ with $d \in \{2,3\}$.
 The weak lower semicontinuity of $I$ plays an important role in the direct method of calculus of variations. Under certain growth conditions from
	above and from below, the quasiconvexity of $W$ is equivalent to the weak lower semicontinuity of $I$ \cite{Mor52, Mor66}, see also \cite{Dac82} for a detailed overview. Without the growth condition from above, there are only a few results on the lower semicontinuity for quasiconvex integrands available \cite{Kri15}. 
	However, such a condition is very restrictive in finite strain theory, where often energies are used with $W(F)= \infty$ if $\det(F) \leq 0$ and $W(F) \to \infty$ for $\det(F) \to 0$.
	
	A sufficient condition for quasiconvexity is presented in \cite{Mor52}, which is equivalent to polyconvexity. One of the main advantages of polyconvexity is that no growth condition from above is necessary to deduce the weak lower semicontinuity of $I$, cf.~\cite{Bal77}, \cite{Bal77b}, also see \cite{Bal02} and \cite{Dac08} for an overview. Thus, polyconvexity is well suited for many problems in elasticity.
	\begin{definition}[Polyconvexity for functions on $\R^{d \times d}$]For $d\in \N$ a function $W \colon \R^{d \times d} \to \R_\infty$ is said to be polyconvex if there exists a convex function $G \colon \R^{K_d} \to \R_\infty$ such that
		\begin{equation}\label{eq:def:Polyconvexity}
			W(F) = G(\M(F))  \qquad \text{for all }F \in \R^{d \times d}\,,
		\end{equation}
		where $\M(F) \in \R^{K_d}$ with $K_d \coloneqq \sum_{i=1}^d \binom{d}{i}^2$ denotes the vector of all minors of $F \in \R^{d \times d}$.
		If $G$ can additionally be chosen to be lower semicontinuous, we call $W$ lower semicontinuous polyconvex. 
	\end{definition}
	The vector of minors can be identified for $d = 2$ by $\M(F) = \left(F,\, \det(F)\right) \in \R^5$ and for $d=3$ by $\M(F) = \left(F, \,\adj(F)^\top,\, \det(F)\right) \in \R^{19}$.
	Note that lower semicontinuous polyconvexity is the natural extension of the notion of polyconvexity for functions that can attain the value $\infty$, cf.~Remark \ref{rem:lsc} below for more details.
	
	Many highly relevant energy densities in finite strain theory are isotropic, such as the neo-Hookean, Saint Venant--Kirchhoff, Odgen hyperelastic and Mooney--Rivlin models.
	\begin{definition}[Isotropic functions]\label{def:Isotrop}
For $d\in \N$, we call a function $W \colon \R^{d \times d} \to \R_\infty$ isotropic if
		\begin{equation}\label{eq:Iso}
			W(F) = W(R_1 F R_2) \qquad \textrm{ for all } F \in \R^{d \times d}, \, R_1, R_2 \in \SO(d)\,,
		\end{equation}
		where $\SO(d) \coloneqq \{Q \in \O(d) \mid \det(Q) = 1\}$ denotes the special orthogonal group which is a subset of the orthogonal group $\O(d) \coloneqq \{Q \in \R^{d \times d} \mid Q^\top Q = \1\}$.
	\end{definition}
	
		\subsection{Main result}
	We aim to combine the notion of isotropy with polyconvexity, leading to a dimension reduction.
	As described in more detail below, one can decompose every matrix $F\in \R^{d\times d}$ into a product of the form
	\begin{equation}
		F= R_1 \diag(\nu) R_2
	\end{equation} for $R_1, R_2\in \SO(d)$ and a diagonal matrix $\diag(\nu)$ with entries $\nu \in \R^d$. Together with \eqref{eq:Iso}, we deduce that every isotropic function $W$ is completely defined by its values on the set of diagonal matrices. This leads to the identification of isotropic functions with $\Pi(d)$-invariant functions $\Phi \colon \R^d \to \R_\infty$ via \begin{equation}\label{eq:Phi=Wdiag}
		\Phi(\nu) = W(\diag(\nu)) \qquad \nu \in \R^d \,,
	\end{equation} cf.~Definition \ref{def:Pid-invariant} and 
	Lemma~\ref{lem:Ident-Iso-Pi_d}. The notion of polyconvexity can be transferred to functions on $\R^d$:
	\begin{definition}[Polyconvexity for functions on $\R^{d}$]\label{def:PolyVector}
		A function $\Phi \colon \R^{d} \to \R_\infty$ is said to be polyconvex if there exists a convex function $g \colon \R^{k_d} \to \R_\infty$ such that
		\begin{equation}\label{eq:def:Polyconvexity-Vector}
			\Phi(\nu) = g(\m(\nu))  \qquad \text{for all } \nu \in \R^{d}\,,
		\end{equation}
		where $\m(\nu) \in \R^{k_d}$ and $k_d \coloneqq 2^d-1$ denotes the vector of all minors of $\nu \in \R^d$, i.e.~
		\begin{align}\label{eq:def:m23}
			&\m(\nu) = (\nu_1, \,\nu_2, \,\nu_1 \nu_2)\in \R^3 &&\textrm{ for } d = 2 \,,
			\\
			&\m(\nu) = (\nu_1, \, \nu_2, \, \nu_3, \, \nu_2 \nu_3, \,\nu_1 \nu_3,\,\nu_1 \nu_2, \,\nu_1 \nu_2 \nu_3)\in \R^{7} &&\textrm{ for } d = 3\,\,.
		\end{align}
		For arbitrary dimension, $\m$ is given in \eqref{eq:def:minorsofVector}.
		If $g$ can additionally be chosen lower semicontinuous, we call $\Phi$ lower semicontinuous polyconvex. 
	\end{definition}
	
	Our main result shows that the dimension reduction for isotropic functions can be transferred to polyconvexity:
	\begin{theorem}[Dimension reduction for polyconvexity of isotropic functions]\label{thm:thm-for-diagonal}
		Let $d \in \{2,3\}$ and $W \colon \R^{d \times d} \to \R_\infty$ be isotropic and $\Phi \colon \R^{d } \to \R_\infty$ be given by \eqref{eq:Phi=Wdiag}. Then, the following statements are equivalent:
		
		(i) $W$ is lower semicontinuous polyconvex,
		
		(ii) there exists a lower semicontinuous convex function ${g \colon \R^{k_d} \to \R_\infty}$  such that
		\begin{equation}\label{eq:Polyconvexity-m}
			W(\diag(\nu)) = g(\m(\nu)) \qquad \text{for all }\nu \in \R^{d}\,,
		\end{equation}
		
		(iii) $\Phi$ is lower semicontinuous polyconvex (in the sense of Definition \ref{def:PolyVector}).
	\end{theorem}
	\smallskip 
	
	The minors of vectors $\m$ are related to the minors of quadratic matrices $\M$ by means of diagonal matrices: The entries of $\m(\nu)$ are those entries of $\M(\diag(\nu))$, which are in general  non-trivial, cf.~
	\begin{align}
		\begin{aligned}\label{eq:M(diag)}
			\M(\diag(\nu)) &= (\diag(\nu), \, \nu_1 \nu_2 ) &&\textrm{ for } d=2\,,
			\\
			\M(\diag(\nu)) &= (\diag(\nu), \, \diag(\nu_2\nu_3,\, \nu_1 \nu_3, \,\nu_1 \nu_2),\, \nu_1 \nu_2 \nu_3) &&\textrm{ for } d = 3\,.
		\end{aligned}
	\end{align}
	Theorem \ref{thm:thm-for-diagonal} yields a dimension reduction for the convex representative $G \colon \R^{K_d} \to \R_\infty$ of  \eqref{eq:def:Polyconvexity} to a convex representative $g \colon \R^{k_d} \to \R_\infty$ in \eqref{eq:def:Polyconvexity-Vector}, where $K_2 = 5$, $K_3= 19$ and $k_2 =3$, $k_3=7$.
	
	In Theorem \ref{thm:Phi=g}, we formulate Theorem \ref{thm:thm-for-diagonal} in terms of the signed singular values, which is closer to the spirit of the dimension reduction result of Ball \cite[Theorem 5.2]{Bal77}. 
	The equivalence of our main result encompasses two implications: One direction, namely the existence of $g$, can be easily inferred from the polyconvexity of $W$ through a restriction of the convex representative $G$. The proof of the other direction is more involved and the main focus of this paper.
	
	Theorem \ref{thm:thm-for-diagonal} implies some symmetry condition on $\Phi$ and $g$, namely $\Pi(d)$-invariance and poly-$\Pi(d)$-invariance, respectively, which is discussed below in Section \ref{sec:PolyInvariance}. 
	Finally, we note that or finite-valued functions, we can deduce with Remark \ref{rem:lsc} that Theorem \ref{thm:thm-for-diagonal} remains true without the lower semicontinuity of the convex representatives:
	\begin{corollary}
		[Dimension reduction for polyconvexity of finite isotropic functions]\label{cor:thm-for-finitevalue}
		Let $d \in \{2,3\}$ and $W \colon \R^{d \times d} \to \R$ be isotropic and $\Phi \colon \R^{d } \to \R$ be given by \eqref{eq:Phi=Wdiag}. Then, the following statements are equivalent:
		
		(i) $W$ is polyconvex,
		
		(ii) there exists a convex function ${g \colon \R^{k_d} \to \R}$  such that \eqref{eq:Polyconvexity-m} is satisfied,
		
		(iii) $\Phi$ is polyconvex (in the sense of Definition \ref{def:PolyVector}).
	\end{corollary}
	A proof of Corollary \ref{cor:thm-for-finitevalue} is given at the end of Section \ref{sec:Duality}.
	\begin{remark}[On the lower semicontinuity of the convex representative]\label{rem:lsc}
		The notion of lower semicontinuous polyconvexity is a natural extension of polyconvexity for functions that attain the value $\infty$.
		Proofs for the weakly lower semicontinuity of $I$ (\cite{Bal77, Bal77b}, see also \cite{BCO81} and  \cite{Cia88, Dac08, Rin18}) employ not only the polyconvexity of $W$ but also the lower semicontinuity of the convex representative.
		Note that the convex representative cannot always be chosen with the same regularity as the function itself \cite{Bev03, Bev11}.
		
		Every lower semicontinuous polyconvex function is also polyconvex. The reverse direction holds for functions that attain only finite values, i.e.~$W \colon \R^{d \times d} \to \R$.
		In this case, there exists a convex representative $G \colon \R^d \to \R$ that also attains only finite values, cf.~\cite[Theorem 5.6]{Dac08}.
		Since every convex function with finite values is continuous (cf.~\cite[Theorem 2.31]{Dac08}), the convex representative of $W$ is continuous  and thus $W$ is polyconvex.
		The same argumentation can be transferred to the function $\Phi$.
	\end{remark}
	
	\subsection{Motivation and overview on the literature}
	The dimension reduction for polyconvexity in terms of the (singed) singular values plays a crucial role in the concept of polyconvexity. 
	John Ball derived a sufficient, but not necessary, criterion for the polyconvexity of isotropic functions in terms of the singular values (cf. \cite[Theorem 5.2]{Bal77}). He used this criterion to show the polyconvexity of some Odgen-type energy densities in \cite{Bal77}. 
	His criterion has proven to be particularly useful for the construction of polyconvex functions and is employed in the verification of other polyconvexity criteria for functions formulated in terms of principal stretches and invariants, e.g.~\cite{CDHL88, Ste03}.
	Compared to the criterion in \cite[Theorem 5.2]{Bal77}, our criterion in Theorem \ref{thm:thm-for-diagonal} is not only sufficient but also necessary, thereby expanding the range of applications to the following points:
	(i) The possibility to apply the criterion for verifying the polyconvexity of any isotropic polyconvex function.
	(ii) The verification of non-polyconvexity for isotropic functions.
	(iii) The construction of polyconvex envelopes and optimal fitting of polyconvex functions in the case of isotropy.
	
	\smallskip
	
	An isotropic function is convex if and only if its restriction to the set of diagonal matrices is convex (see Proposition \ref{prop:ConvexityISO}).
	Historically, this result was first shown for $\O(d) \times \O(d)$-times invariant functions (functions that satisfy \eqref{eq:Iso} also for $R_1, R_2 \in \O(d)$) and the result was formulated in terms of the singular values instead of diagonal matrices in \cite{TF71, Hil70}. In \cite{Bal77, LeD90}, this result was shown  for $\O(d) \times \O(d)$-times invariant functions based on the von Neumann inequality \cite{VN37, Mir75}. For isotropic, i.e.~$\SO(d) \times \SO(d)$-times invariant functions, the result was established for $d=2$ in \cite{DK93}. For $d\geq 3$, it was shown in \cite{Vin97} based on a convexity theorem of \cite{Kos73}. A different proof based on a generalized von Neumann inequality and signed singular values was given in \cite{Ros97, DM07}.

	Ball used his convexity result for $\O(d) \times \O(d)$ invariant functions to derive his well-established criterion for the polyconvexity of isotropic functions \cite[Theorem 5.2]{Bal77}. Indeed, the additional invariance simplifies the derivation; however, it is also one of the reasons why it does not yield a necessary criterion.
	For $d=3$, he applied the convexity criterion of $\O(d) \times \O(d)$-invariant functions for the first nine arguments of $G$, which belong to $F$, and the second nine arguments, which belong to $\adj(F)$, separately. 
	This introduces additional restrictive symmetries, which are not even necessary for the polyconvexity of $\O(d) \times \O(d)$-invariant functions.
	
	In \cite{BDG94}, the polyconvex envelope of $\O(2)\times \O(2)$-invariant functions was characterized by means of the singular values. This characterization provides an implicit criterion for a function to be polyconvex. 
	An only sufficient but less restrictive criterion than that of \cite{Bal77} was presented in \cite{RS94} for the polyconvexity of isotropic functions restricted to the set of matrices with positive determinant and $d =2$. Later, for $d=2$, sufficient and necessary criteria for the polyconvexity of isotropic functions were shown in \cite{Ros97, Sil99}, requiring only the non-decreasing property along several lines. Moreover, for $d=3$ the polyconvexity criterion of \cite{Bal77} was refined into a sufficient and less restrictive but still not necessary criterion in \cite{Sil99}. 
	For $d \in \{2,3\}$, necessary and sufficient conditions for the polyconvexity of isotropic functions defined on $\{F \in \R^{d \times d} \mid \det(F) \geq 0\}$ by means of the singular values were derived in \cite{Mie05}. They are not formulated in the form of the criterion of \cite{Bal77} but in a more implicit way. 
	Afterwards, in \cite{DM06} for $d=2$, the notion of polyconvexity was reduced for isotropic function to the set of diagonal matrices or the signed singular values. This result corresponds to our main result for $d=2$.
	Compared to the previous necessary and sufficient results, it has the advantage of not requiring any complicated-to-access non-decreasing properties along certain lines; instead, it relies solely on the convexity of the function $g$.
	Hence, in the isotropic setting, the polyconvexity criterion of \cite{DM06} for $d=2$ and our Theorem \ref{thm:thm-for-diagonal} for $d\in \{2,3\}$ can be used directly for a dimension reduction of already existing methods in the field of polyconvexity, for instance for determining whether a function is polyconvex, for computing the polyconvex envelope or for fitting a polyconvex function. Such methods can be found for example in \cite{Dac08, DH96, DH98, Dac87, Bar05}.
	At this point, it is essential that the polyconvex envelope of $\O(d)\times \O(d)$-invariant or isotropic functions remains $\O(d)\times \O(d)$-invariant or isotropic, respectively, which was shown in \cite{BDG94, Dac08}. 
	
	\smallskip
	
	For the related notion of rank-1 convexity, we refer to \cite{Sil99, Sil99b, Sil03, Dac01, Aub95, Dav91, GK22} for investigations on isotropic functions.
	However, the notion of rank-1 convexity or quasiconvexity cannot be restricted for isotropic functions to diagonal matrices, which was shown in \cite{DK93, Mue99b}.
	Furthermore, investigation on the polyconvexity and rank-1 convexity for isotropic and, additionally, isochoric functions can be found in \cite{Mie05, MGIN17, GIMN18} and a characterization of symmetric polyconvexity was given in \cite{BKS19}. The non-locality of polyconvexity is investigated in \cite{Kri00}.
	Using diagonal matrices, it was showed that the Saint Venant--Kirchhoff model is not polyconvex \cite{Ra86}.

	\smallskip
		
	In this work, we order for $d=3$ the minors of a matrix $F \in \R^{d \times d}$ by $\M(F) = \left(F, \,\adj(F)^\top, \,\det(F)\right)$, while it is more common in the literature to arrange them by
	$\left(F,\, \adj(F),\, \det(F)\right)$. Of course, this order does not affect polyconvexity; we use our notation solely because it is more convenient in one estimate.
	
	\smallskip
	
	\textbf{Organization of the paper.}
	This paper is mainly devoted to the presentation and proof of Theorem \ref{thm:thm-for-diagonal}. Statement (iii) is only a reformulation of (ii). The implication (i) $\Rightarrow$ (ii) is trivial, where $g$ can be obtained by restricting the convex representative $G$ to the image of $\m$.
	To show the crucial implication ``(iii) $\Rightarrow$ (i)'', we use the polyconvex conjugation, which is an extension of the Fenchel--Legendre conjugation to the setting of polyconvexity. It allows us to trace the implication back to the same investigation but for elementary functions, cf.~Lemma \ref{lem:B=D}. For this, we use a polynomial von Neumann type inequality based on an result of \cite{Mie05}.
	
	What is more, isotropic energy densities in elasticity are often formulated in terms of the principal matrix invariants of the right or left stretch tensor.
	In \cite{Ste03}, the polyconvexity criterion for isotropic functions of \cite{Bal77} was used to derive a sufficient polyconvexity criterion for functions formulated in terms of these invariants.
	We transfer this argumentation to our stronger polyconvexity criterion and, thus, obtain a new sufficient criterion for the polyconvexity of a function, which is formulated in terms of the elementary symmetric polynomials of the signed singular values of $F$ or the principal matrix invariants of the right or left stretch tensor. 
	
	The paper is organized as follows. In Section \ref{sec:BasicDef}, we recap some general results on signed singular values, the identification of isotropic and $\Pi(d)$-invariant functions and the dimension reduction for isotropic convex functions. Moreover, we discuss some invariance implications for the function $g$ of Theorem \ref{thm:thm-for-diagonal}.
	In Section \ref{sec:ConvecConjugation} , we present some elementary results on the Fenchel--Legendre conjugation, its generalization to polyconvexity and the identification of both.
	In Section~\ref{sec:Duality}, we use the polyconvex conjugation in the space of singular values to show the main results.
	In Section~\ref{sec:Invariants}, we use this characterization of polyconvexity in order to make some remarks on polyconvexity with respect to matrix invariants.
			
	\section{Singed singular values and isotropy}\label{sec:BasicDef}
	\subsection{Singed singular value decomposition}\label{ssec:SSV}
	For arbitrary $d \in \N$, every matrix $F \in \R^{d \times d}$ can be decomposed with the singular value decomposition into the product of a diagonal matrix with non-negative entries and two orthogonal matrices, i.e.~there exists $\tilde{R}_1, \tilde{R}_2 \in \O(d)$ and $\tilde{\nu} \in [0, \infty)$ such that
			\begin{equation}\label{eq:SV-decomposition}
				F = \tilde{R}_1 \diag(\tilde{\nu}) \tilde{R}_2 \,,
			\end{equation}
			where $\diag(\tilde{\nu}) \in \R^{d \times d}$ denotes the diagonal matrix with entries $\tilde{\nu}_1, \dots, \tilde{\nu}_d \in [0, \infty)$. 
			We call a vector $\tilde{\nu} \in [0, \infty)^d$ resulting from the decomposition \eqref{eq:SV-decomposition} singular values of $F$.
			
			The singular value decomposition is not directly compatible with the concept of isotropy since \eqref{eq:Iso} holds only for $R_1, R_2 \in \SO(d)$ but not for arbitrary elements in $\O(d)$. This leads to the singed singular value decomposition, a refinement of the singular value decomposition. Essentially one allows that the singular values $\tilde{\nu}$ can be negative so that one can choose $R_1$ and $R_2$ with determinant $1$:
			
			For arbitrary $d \in \N$, every matrix $F \in \R^{d \times d}$ can be decomposed in the product of a diagonal matrix and two orthogonal matrices with determinant one, i.e.~there exists $R_1, R_2\in \SO(d)$ and $\nu \in \R^d$ such that
			\begin{equation}\label{eq:SSV-decomposition}
				F = R_1 \diag(\nu) R_2 \,.
			\end{equation}
			We call a vector $\nu \in \R^d$ resulting from the decomposition \eqref{eq:SSV-decomposition} signed singular values of $F$.
			A singed singular value decomposition can be constructed with the singular value decomposition as follows: Let $\tilde{R}_1, \tilde{R}_2 \in \O(d)$ and $\tilde{\nu} \in [0, \infty)^d$ be given by the singular value decomposition of $F$, i.e.~they satisfy \eqref{eq:SV-decomposition}.
			We choose $R_1 = \tilde{R}_1 \diag(\det(R_1),1,\dots, 1) \in \SO(d)$, $R_2 = \diag(\det(R_2),1,\dots, 1) \tilde{R}_2 \in \SO(d)$ and $\nu = (\det(R_1)\det(R_2) \tilde{\nu}_1, \, \dots,\tilde{\nu}_d) \in \R^d$ and deduce with \eqref{eq:SV-decomposition} that the signed singular value decomposition \eqref{eq:SSV-decomposition} holds.
			The signed singular value decomposition \eqref{eq:SSV-decomposition} is not unique. One can change the order of the entries of $\nu$ arbitrarily and can also change the signs for an even number of the entries $\nu$ (by adjusting $R_1$ and $R_2$ accordingly).
			
			To be more precise: let $\S(d)$ be the symmetric group (which we identify with the corresponding permutation matrices, wherever it is convenient) and
			\begin{equation*}
				\Pi(d) \coloneqq \left\{S \diag(\epsilon) \mid S \in \S(d),\, \epsilon \in \{-1,1\}^{d},\, \prod_{i=1}^d \epsilon_i = 1\right\}\,.
			\end{equation*}
			The signed singular values are unique up to a multiplication by an element in $\Pi(d)$, i.e.~$\nu$ is a signed singular value of $F\in \R^{d \times d}$ if and only if $S\nu$ is a signed singular value of $F$ for all $S \in \Pi(d)$.
			
			In the context of elasticity theory, the signed singular values correspond, up to the sign, to the singular values of the deformation gradient $F$, which are equal to the eigenvalues of the symmetric and positive  semidefinite right and left stretch tensors $U$, $V$, which are defined via the right and left Cauchy--Green tensors $C$, $B$ given by $C = U^2= F^\top F$ and $B = V^2 = F F^\top$.  We call a matrix $F \in \R^{d \times d}$ symmetric if $F = F^\top$ and positive semidefinite if $x^\top F x\geq 0$ for all $x \in \R^d$. 
			Due to their physical meaning, isotropic energy densities are often more naturally expressed in terms of the singular values or signed singular values and thus in terms of $\nu$ rather than in terms of $F$. 
						
			Theorem \ref{thm:thm-for-diagonal} can be reformulated using the concept of signed singular values:
			\begin{theorem}[Polyconvexity in terms of the signed singular values]\label{thm:Phi=g}
				Let $d \in \{2,3\}$. An isotropic function $W \colon \R^{d \times d} \to \R_\infty$ is lower semicontinuous polyconvex if and only if
				there exists a lower semicontinuous convex function ${g \colon \R^{k_d} \to \R_\infty}$ with $k_d \coloneqq 2^d-1$ such that 
				\begin{align}
					\begin{aligned}\label{eq:W(F)=g}
						W(F) &= g(\m(\nu)) = g(\nu_1, \, \nu_2, \,\nu_1 \nu_2) &&\textrm{ for } d = 2,
						\\
						W(F) &= g(\m(\nu)) = g(\nu_1, \, \nu_2, \, \nu_3, \, \nu_2 \nu_3, \,\nu_1 \nu_3,\,\nu_1 \nu_2, \,\nu_1 \nu_2 \nu_3) &&\textrm{ for } d = 3
					\end{aligned}
				\end{align}
				for all $F \in \R^{d \times d}$ and for all choices of signed singular values $\nu \in \R^d$ of $F$.
			\end{theorem}

			\subsection{Isotropic and \texorpdfstring{$\O(d)\times \O(d)$}{O(d)xO(d)}-times invariant functions}\label{ssec:IsoFunctions}
			A function \linebreak \text{$W\colon \R^{d\times d} \to \R_\infty$} is called $\O(d)\times \O(d)$-times invariant if 
			\begin{equation}\label{eq:OdxOd-Inv}
				W(F) = W(R_1 F R_2) \qquad \textrm{ for all } F \in \R^{d \times d}, \, R_1, R_2 \in \O(d)\,.
			\end{equation}
			A function $W$ is said to be $\SO(d)\times \SO(d)$-times invariant if \eqref{eq:Iso} is satisfied, i.e.~\eqref{eq:OdxOd-Inv} is satisfied for $R_1,R_2 \in \SO(d)$. Thus,  $\O(d)\times \O(d)$-times invariance is a stronger condition than $\SO(d)\times \SO(d)$-times invariance.
			\begin{remark}[Alternative definition of isotropy]
			In Definition \ref{def:Isotrop}, we have defined the term isotropic as 
$\SO(d)\times \SO(d)$-times invariance. Sometimes in the literature (e.g.~\cite{Bal77}), the condition isotropic is defined in the following equivalent manner:
A function $W \colon\R^{d\times d} \to \R_\infty$ is said to be objective if 
\begin{equation*}
	W(F) = W(RF) \qquad\textrm{for all } F \in \R^{d \times d}, \, R \in \SO(d)\,.
\end{equation*}
A function $W\colon  \R^{d\times d} \to \R_\infty$ is called isotropic if it is objective and 
\begin{equation*}
	W(F) = W(RFR^\top) \qquad\textrm{for all } F \in \R^{d \times d}, \, R \in \O(d)\,.
\end{equation*}
\end{remark}

			Using the singed singular value decomposition \eqref{eq:SSV-decomposition} and the $\SO(d)\times \SO(d)$ invariance, we observe that 
			\begin{equation}
				W(F) = W(R_1 \diag(\nu) R_2) = W(\diag(\nu))\,.
			\end{equation}
			Thus, every isotropic function is already completely prescribed by its values on the set of diagonal matrices.
			We define for a given function $W \colon \R^{d \times d} \to \R_\infty$ the function $\Phi \colon \R^{d} \to \R_\infty$ by $\Phi = W \circ \diag$.
			The isotropy of $W$ implies that the corresponding function $\Phi$ is $\Pi(d)$-invariant: 
			\begin{definition}[$\Pi(d)$-invariance]\label{def:Pid-invariant}
			For $d\in \N$, a function $\Phi \colon \R^d \to \R_\infty$ is called $\Pi(d)$-invariant if
				\begin{align*}
					\Phi(\nu) &= \Phi(S \nu)  && \hspace{-3cm}\text{for all } \nu \in \R^d\,,\, S \in \Pi(3)\,, \text{i.e.~}
					\\
					\Phi(\nu_1, \dots, \nu_i, \dots , \nu_j, \dots, \nu_d) &= \Phi(\nu_1, \dots, \nu_j, \dots , \nu_i, \dots, \nu_d) &&\forall \nu \in \R^d,\,  i < j\,,
					\\
					\Phi(\nu_1, \dots, \nu_i, \dots , \nu_j, \dots, \nu_d) &= \Phi(\nu_1, \dots, -\nu_i, \dots , -\nu_j, \dots, \nu_d) &&\forall \nu \in \R^d,\,  i < j  \,.
				\end{align*}
			\end{definition}
			As immediate consequence of the signed singular value decomposition, we obtain the following lemma (see also \cite[Proposition 5.31]{Dac08}):
			\begin{lemma}[Identification of isotropic and $\Pi(d)$-invariant functions]\label{lem:Ident-Iso-Pi_d}
				For $d\in \N$, the space of isotropic functions can be identified with the space of $\Pi(d)$-invariant functions via the following bijectivity: 
				\begin{align*}
					\{ W\colon \R^{d \times d} \to \R_\infty\mid W \text{ is isotropic }\}
					&\longleftrightarrow
					\{ \Phi \colon \R^{d} \to \R_\infty\mid \Phi \text{ is } \Pi(d)\text{-invariant}\}\,,
					\\
					W \qquad &\longmapsto \quad (\nu \mapsto \Phi(\nu) \coloneqq W(\diag(\nu))\,,
					\\
					(F \mapsto W(F) \coloneqq \Phi(\nu_F))  \quad &\longmapsfrom \qquad  \Phi \,,
				\end{align*}
				where $\nu_F$ denote some choice of singed singular values of $F$.
			\end{lemma}
			
			So, we can write in particular for arbitrary $F \in \R^{d \times d}$ and arbitrary signed singular values $\nu_F$ of $F$:
			\begin{equation}\label{eq:W=Phi}
				W(F) = W(\diag(\nu_F)) = \Phi(\nu_F) \,.
			\end{equation}
			
			The above identification of isotropic functions with $\Pi(d)$-invariant functions can be refined when considering the smaller class of $\O(d) \times \O(d)$ invariant functions. Such functions can be identified with $\S(d)$-invariant functions $\tilde{\Phi} \colon [0, \infty)^d \to \R_\infty$. A function $\tilde{\Phi}\colon [0, \infty) \to \R_\infty$ is called $\S(d)$-invariant if 
			\begin{equation*}
				\tilde{\Phi}(\nu) = \tilde{\Phi}(S \nu) \qquad \text{for all } \nu \in [0, \infty) \,,\, S \in \S(d)\,,
			\end{equation*} 
			where we identified the permutation $S \in \S(d)$ with the permutation matrix, i.e.~
			$S\nu= (\nu_{S(1)}, \dots, \nu_{S(d)})$ for $\nu \in \R^d$.
			Using the singular value decomposition instead of the singed singular value decomposition, we can adjust Lemma \ref{lem:Ident-Iso-Pi_d} as follows: 
			\begin{lemma}[Identification of $\O(d) \times \O(d)$- and $\S(d)$-invariant functions]\label{lem:Ident-Od-Sd}
			For $d\in \N$ the space of $\O(d) \times \O(d)$ functions can be identified with the space of $\S(d)$-invariant functions via the following bijectivity: 
				\begin{align*}
					\left\{\parbox{4.8cm}{
						$W\colon \R^{d \times d} \to \R_\infty\mid$\newline
						\phantom{\quad}$W$ is $\O(d) \times \O(d)$-invariant
					}\right\}
					&\longleftrightarrow
					\{\tilde{\Phi} \colon [0,\infty)^d \to \R_\infty\mid \tilde{\Phi} \text{ is } \S(d)\text{-invariant}\}\,,
					\\
					W \qquad &\longmapsto \quad (\tilde{\nu} \mapsto \tilde{\Phi}(\tilde{\nu}) \coloneqq W(\diag(\tilde{\nu}))\,,
					\\
					(F \mapsto W(F) \coloneqq \tilde{\Phi}(\tilde{\nu}_F) ) \quad &\longmapsfrom \qquad  \tilde{\Phi} \,,
				\end{align*}
				where $\tilde{\nu}_F$ denotes some choice of singular values of $F$.
			\end{lemma}
			
			We note that one can also formulate Lemma \ref{lem:Ident-Od-Sd} with functions $\Phi$ defined on the entire $\R^d$ instead of $\R^d$, when assuming additionally to the $\S(d)$-invariance that $\Phi(\nu)= \Phi(\diag(\epsilon)\nu)$ for all $\nu \in \R^d$ and $\epsilon \in \{-1,1\}^d$.

			In many applications, $W$ is isotropic and attains the value $\infty$ when $\det(F) \leq 0$. In this case, one can also formulate $W$ by means of the singular values using the following separation of cases:
			\begin{equation}\label{eq:IsoSingularValues}
				W(F) = \begin{cases}
					\tilde{\Phi}(\tilde{\nu}_F) & \text{ if } \det(F)>0 \text{ for some singular values } \tilde{\nu}_F \text{ of } F\,,  \\
					\infty &\text{ if } \det(F)\leq 0
				\end{cases}
			\end{equation}
			
			\subsection{Dimension reduction for convex functions}
			An isotropic function is convex if and only if the corresponding $\Pi(d)$-invariant function is convex.
			\begin{proposition}[Convexity of isotropic functions]\label{prop:ConvexityISO}
				For $d \in \N$, let $W \colon \R^{d \times d} \to \R_\infty$ be isotropic and $\Phi\colon \R^d \to \R_\infty$ the corresponding $\Pi(d)$-invariant function given by Lemma~\ref{lem:Ident-Iso-Pi_d}.
				Then, the following statements are equivalent:
				
				(i) $W$ is convex and lower semicontinuous.
				
				(ii) $\Phi$ is convex and lower semicontinuous.
				
				(iii) $W \circ \diag \colon \R^d \to \R_\infty$ is convex and lower semicontinuous.
			\end{proposition}
			The  implication $(i) \Rightarrow (ii)$ is trivial and the equivalence of (ii) and (iii) is an immediate consequence of $\Phi = W\circ \diag$ given by Lemma \ref{lem:Ident-Iso-Pi_d}. For a proof of the non-trivial direction, see for instance \cite{Dac08}.
			
			Since every $\O(d) \times \O(d)$-invariant function is also $\SO(d) \times \SO(d)$-invariant, Proposition~\ref{prop:ConvexityISO} can also be applied for those functions. The corresponding function $\Phi$ additionally satisfies the symmetry $\Phi(\nu)= \Phi(\diag(\epsilon)\nu)$ for all $\nu \in \R^d$ and $\epsilon \in \{-1,1\}^d$, which can be replaced when considering only $\tilde{\Phi} = \Phi|_{[0,\infty)^d}$ and assuming that $\tilde{\Phi}$ is non-decreasing in every argument:
			\begin{lemma}[Convexity of $\O(d) \times \O(d)$-invariant functions]
			For $d\in \N$, let $W \colon \R^{d \times d} \to \R_\infty$ be $\O(d) \times \O(d)$-invariant and $\tilde{\Phi}$ the corresponding $\S(d)$-invariant function given by Lemma \ref{lem:Ident-Od-Sd}.
				Then, the following statements are equivalent:
				
				(i) $W$ is convex and lower semicontinuous.
				
				(ii) $\tilde{\Phi}$ is convex, lower semicontinuous  and non-decreasing in every variable.
			\end{lemma}
			
			\subsection{Polyconvexity and isotropy}\label{sec:PolyInvariance}
			We introduced the notion of polyconvexity for functions on $\R^d$ in Definition \ref{def:PolyVector} using the minors of vectors. For $d\in \{2,3\}$, we defined the minor $\m(\nu)$ of a vector  $\nu \in \R^d$ in \eqref{eq:def:m23}.
			For general dimension $d \in \R$, we define $\m(\nu)$  by the (in general) non-vanishing entries of $\M(\diag(\nu))$:
			\begin{equation}\label{eq:def:minorsofVector}
				\m \colon \R^d \to \R^{k_d}\,,\quad \nu \mapsto \m(\nu) \coloneqq \left(\prod\limits_{j \in K}   \nu_j \right)_{K \in \mathcal{P}(\{1,\dots, d\})}   \quad (\text{with }k_d \coloneqq 2^d -1)\,,
			\end{equation}
			where $\mathcal{P}(\cdot)$ denotes the power set.

			The classical concept of polyconvexity for functions on the set of matrices is not only restricted to quadratic matrices but is also defined on the set $\R^{1 \times d}$ and $\R^{d \times 1}$, where it is equivalent to classical convexity. In particular, it does not coincide with the notion of polyconvexity from Definition \ref{def:PolyVector}.
			
			For $\Pi(d)$-invariant functions, the following notion of \textit{signed singular value polyconvexity} becomes helpful, which we define via the polyconvexity of the corresponding isotropic function $W$:
			\begin{definition}[Singed singular value polyconvex functions]
			For $d\in \R^d$, a $\Pi(d)$-invariant function $\Phi \colon \R^d \to \R_\infty$ is called signed singular value polyconvex if the corresponding function $W \colon \R^{d \times d} \to \R_\infty$, given by Lemma \ref{lem:Ident-Iso-Pi_d}, is polyconvex.
				We call $\Phi$ lower semicontinuous signed singular value polyconvex if $W$ is lower semicontinous polyconvex.
			\end{definition}
			Using this definition, we can reformulate Theorem \ref{thm:thm-for-diagonal} as follows: A $\Pi(d)$-invariant function $\Phi \colon \R^d \to \R_\infty$ is lower semicontinuous signed singular value polyconvex if and only if it is lower semicontinuous polyconvex.
			One direction is trivial and given in the following Lemma while the reverse direction is more involved and presented in the proof of Theorem \ref{thm:thm-for-diagonal} at the end of Section \ref{sec:Duality}.
			
			\begin{lemma}[Singed singular value polyconvexity implies polyconvexity]\label{lem:EquivPolyconvexityPhi}
				For $d\in \N$, let $\Phi \colon \R^d \to \R_\infty$ be $\Pi(d)$-invariant and lower semicontinuous signed singular value polyconvex. Then, $\Phi$ is lower semicontinuous polyconvex.
			\end{lemma}
			\begin{proof}
				Let $W$ be the isotropic representative of $\Phi$ and $G$ be given as the lower semicontinuous convex representative of $W$, i.e.~$\Phi(\nu) = W(\diag(\nu))= G(\M(\diag(\nu)))$ for all $\nu \in \R^d$. We define $g$ as the restriction of $G$ to the in general non-trivial values of $\M(\diag(\nu) = \m(\nu)$, which preserves the convexity and lower semicontinuity, and thus $\Phi(\nu) = g(\m(\nu))$ is lower semicontinuous polyconvex.
			\end{proof}

			The function $G$ in the definition \eqref{eq:def:Polyconvexity} and analogously $g$ in \eqref{eq:def:Polyconvexity-Vector} is not unique. However, every choice of these functions has to be poly-isotropic or poly-$\Pi(d)$-invariant, when $W$ or $\Phi$ is isotropic or $\Pi(d)$-invariant, respectively.
			\begin{definition}[Poly-isotropic and poly-$\Pi(d)$-invariant functions]
				Let $d \in \N$. We call a function $G \colon \R^{d \times d} \to \R_\infty$ poly-isotropic if 
				\begin{equation*}
					G(\M(F)) = G(\M (R_1 F R_2)) \qquad \text{for all } F \in \R^{d\times d} \text{ and all } R_1,R_2 \in \SO(d)\,. 
				\end{equation*}
				We call $g \colon \R^d \to \R_\infty$ poly-$\Pi(d)$-invariant if 
				\begin{equation*}
					g(\m(\nu)) = g(\m (S \nu)) \qquad \text{for all } \nu \in \R^d \text{ and all } S \in \Pi(d)\,. 
				\end{equation*}
			\end{definition}
			
			Vice versa, every poly-isotropic function $G$ induces some isotropic function $W$ and every poly-$\Pi(d)$-invariant function $g$ induces some $\Pi(d)$-invariant function $\Phi$. 
			With Theorem \ref{thm:thm-for-diagonal}, we can add polyconvexity to this statement:
			\begin{corollary}[Construction of isotropic polyconvex functions]\label{cor:ConstructionPolyconvex}
				Let $d\in \{2,3\}$ and $g\colon \R^{d}\to \R_\infty$ be lower semicontinuous, convex and poly-$\Pi(d)$ invariant. Then, $W \colon \R^{d \times d}\to \R_\infty$ is well-defined by $
				W(F) \coloneqq g(\m(\nu_F))$  for $\nu_F$ some signed singular values of $F$ and $W$ is lower semicontinuous polyconvex and isotropic.
			\end{corollary}
			Corollary \ref{cor:ConstructionPolyconvex} is particularly useful for the construction of polyconvex functions. 
			In order to construct a poly-$\Pi(d)$-invariant function for $d \in \{2,3\}$, it can be convenient to impose the following stronger symmetry condition:
			\begin{lemma}[Separate $\Pi(d)$-invariance]
				For $d \in \{2,3\}$, let $g\colon \R^{k_d}\to \R_\infty$ satisfy
				\begin{align}
					\begin{aligned}\label{eq:sym:g:full}
						\textrm{if } d = 2\colon \qquad&g(x,\, z) = g(Sx, \, z) && \text{for all } x\in \R^2,\, z \in  \R, \,S \in \Pi(2)\,,
						\\
						\textrm{if } d = 3\colon\qquad&g(x, \, y, \,z) = g(Sx,\, Sy, \,z)  &&\text{for all } x,y\in \R^3,\, z \in  \R,\,S \in \Pi(3)\,.
					\end{aligned}
				\end{align}
				Then, $g$  is poly-$\Pi(d)$-invariant.
			\end{lemma}
			\begin{proof}
				For $d=2$, we obtain the poly-$\Pi(d)$-invariance directly by choosing $x= \nu$ and $z = \nu_1 \nu_2$ in \eqref{eq:sym:g:full}.
				For $d=3$, we choose $x= \nu$, $y=(\nu_2\nu_3,\, \nu_1\nu_3,\, \nu_1 \nu_2)^\top$, $z = x_1x_2x_3$ in \eqref{eq:sym:g:full} and
				employ that
				\begin{equation}\label{eq:Snu=Snu}
					S\left(\begin{array}{c}
						\nu_2\nu_3\\ \nu_1\nu_3\\ \nu_1 \nu_2
					\end{array}\right) = 
					\left(\begin{array}{c}
						(S\nu)_2 (S\nu)_3\\ (S\nu)_1 (S\nu)_3\\
						(S\nu)_1 (S\nu)_2
					\end{array}\right) \qquad \text{for all }\nu \in \R^3,\,S \in \Pi(3)\,.
				\end{equation} 
			\end{proof}
			
			\section{(Poly)convex conjugation}\label{sec:ConvecConjugation}

			For $d \in \N$ and $f \colon \R^d \to \R_\infty$ with  $f \not\equiv \infty$
			the Fenchel--Legendre conjugation, also called convex conjugation, ${f^* \colon \R^d \to \R_\infty}$ of $f$ is given by
			\begin{equation*}
				f^*(y) = \sup\limits_{x \in \R^d} \langle y,x\rangle - f(x),
			\end{equation*}
			where $\langle \cdot, \cdot\rangle$ denotes the Euclidean scalar product. The dual convex conjugation $f^{**} \colon \R^d \to \overline{\R}\coloneqq \R \cup \{\pm \infty\}$ is defined by $f^{**}(x) \coloneqq \sup\limits_{y \in \R^d} \langle y,x\rangle - f^*(y)$.
						
			The following Lemma gives some well-known properties of the convex conjugation, see for instance \cite{Dac08}:
			\begin{lemma}[Basic results on the convex conjugation]\label{lem:LF-Conj}
				For $d \in \N$ and functions $f, g \colon \R^d \to \R_\infty$ with $f,g\not\equiv \infty$ it holds:
				\begin{itemize}
					\item[(i)] $f^*$ and $f^{**}$ are lower semicontinuous and convex.
					\item[(ii)] $f^{**} \leq 
					f$.
					\item[(iii)] 
					if $f \leq g$, there holds: $g^*\geq f$ and $f^{**} \leq g^{**}$.
					\item[(iv)] if $f$ is lower semicontinuous and convex, the identity $f^{**} = f$ holds.
					
				\end{itemize}
			\end{lemma} 
			
			In particular, Lemma \ref{lem:LF-Conj} implies that $f^{**}$ is the lower semicontinuous convex envelope of $f \colon \R^d \to \R_\infty$.
						
			The notion of polyconvexity lifts the set of matrices $\R^{d \times d}$ to $\R^{K_d}$ via $\M$ and employs the convexity in this higher dimensional space. Based on this, the convex conjugation can be extended to a polyconvex conjugation, which was presented in \cite{Dac87} based on \cite{KS83, KS86}. 
			For a function $W\colon \R^{d \times d} \to \R_\infty$ with $W\not \equiv \infty$, its polyconvex conjugation $W^\wedge \colon \R^{K_d} \to \R_\infty$ is defined by
			\begin{equation*}
				W^\wedge(B) \coloneqq \sup \limits_{F \in \R^{d\times d}} \langle B, \M(F) \rangle - W(F) \,.
			\end{equation*}
			We define the dual polyconvex conjugation  $W^{\wedge\vee} \colon \R^{K_d} \to \R_\infty$  by $	W^{\wedge \vee}(F) \coloneqq \sup_{B \in \R^{K_d}} \langle B, \M(F) \rangle - W^\wedge(B)$.
			This polyconvex conjugation can be transferred to functions defined on $\R^d$:
			\begin{definition}[Polyconvex conjugation]
				For $d \in \N$ and $\Phi \colon \R^{d} \to \R_\infty$ with $\Phi \not \equiv \infty$, we define the polyconvex conjugation $\Phi^\wedge \colon \R^{k_d} \to \R_\infty$ by
				\begin{equation*}
					\Phi^{\wedge}(\beta) \coloneqq \sup\limits_{\nu \in \R^d } \langle\beta, \m(\nu)\rangle - \Phi(\nu)\,.
				\end{equation*}
				We define the dual polyconvex conjugation  $\Phi^{\wedge \vee} \colon \R^{k_d} \to \overline{\R}$ by
				$\Phi^{\wedge \vee}  (\nu) \coloneqq \sup_{\beta \in \R^{k_d} } \langle\beta, \m(\nu) \rangle - \Phi^\wedge(\beta)$.
			\end{definition}
			For the sake of clarity, we note that
			\begin{align*}
				&\langle\beta, \m(\nu) \rangle = \beta_1 \nu_1 + \beta_2 \nu_2 + \beta_3 \nu_1 \nu_2 \!&&\textrm{for } d = 2\,,
				\\ 
				&\langle\beta, \m(\nu) \rangle = \beta_1 \nu_1 + \beta_2 \nu_2 + \beta_3 \nu_3 + \beta_4 \nu_2 \nu_3 + \beta_5 \nu_1 \nu_3 + \beta_6 \nu_1 \nu_2+ \beta_7 \nu_1\nu_2\nu_3 \!&&\textrm{for } d = 3\,.
			\end{align*}

			The polyconvex conjugation can be identified with the convex conjugation as follows:
			\begin{lemma}[Identification of the convex and polyconvex conjugation]\label{lem:Equivalence:Conj}
				For $d\in \N$, let $\Phi \colon \R^d \to \R_\infty$ with $\Phi \not \equiv \infty$ and define $h \colon \R^{k_d} \to \R_\infty$ by
				\begin{equation}\label{eq:def:h}
					h(x) \coloneqq \begin{cases}
						\Phi(\nu) &\textrm{ if } x = \m(\nu)\,,
						\\
						\infty &\textrm{ else}.
					\end{cases}
				\end{equation}
				Then, $\Phi^{\wedge}(\beta) = h^*(\beta)$ for all $\beta \in \R^{k_d}$ and $\Phi^{\wedge\vee}(\nu) = h^{**}(\m(\nu))$ for all $\nu \in \R^d$.
			\end{lemma}
			\begin{proof}
				Since $\langle \beta , x \rangle - h(x) = - \infty$ for $x$ not in the image of $\m$, we obtain
				\begin{equation*}
					\Phi^\wedge(\beta) = \sup\limits_{\nu \in \R^{d}} \langle \beta , \m(\nu) \rangle - \Phi(\nu) = 
					\sup\limits_{x \in \R^{k_d}} \langle \beta , x \rangle - h(x) = h^*(\beta)
				\end{equation*}
				for every $\beta \in \R^{k_d}$.
				Inserting this identity into the dual conjugations yields
				\begin{equation*}
					\Phi^{\wedge\vee}(\nu) = \sup\limits_{\beta \in \R^{k_d}} \langle \beta , \m(\nu) \rangle - \Phi^\wedge(\beta) = \sup\limits_{\beta \in \R^{k_d}} \langle \beta , \m(\nu) \rangle - h^*(\beta)
					= h^{**}(\m(\nu)).
				\end{equation*}
			\end{proof}
			We translate Lemma \ref{lem:LF-Conj} to the polyconvex conjugation:
			\begin{proposition}\label{prop:polyconvexConjugation}
				Let $d \in \N$ and $\Phi,\Psi \colon \R^d \to \R_\infty$ with $\Phi, \Psi \not\equiv\infty$. 
				Then, the following statements hold:
				
				(i) $\Phi^{\wedge \vee}$ is lower semicontinuous polyconvex.
				
				(ii) $\Phi^{\wedge \vee} \leq \Phi$.
				
				(iii) if $\Phi \leq \Psi$, then $\Phi^{\wedge} \geq \Psi^{\wedge}$ and $\Phi^{\wedge \vee} \leq \Psi^{\wedge \vee}$.

				(iv) if $\Phi$ is  lower semicontinuous polyconvex, then $\Phi = \Phi^{\wedge \vee}$.
				
			\end{proposition}
			\begin{proof}

				(i):  With Lemma \ref{lem:Equivalence:Conj}, we can identify the convex conjugation with the polyconvex conjugation, i.e.~$\Phi^{\wedge \vee}(\nu) = h^{**}(\m(\nu))$ for $h$ given by \eqref{eq:def:h}. Since $h^{**}$ is lower semicontinuus convex by Lemma \ref{lem:LF-Conj}, we obtain that $\Phi^{\wedge \vee}(\nu)$ is lower semicontinuous polyconvex.
				
				(ii):
				Rearranging the terms in the definition of $\Phi^{\wedge}$ implies $\Phi(\nu) \geq \langle\beta, \m(\nu)\rangle - \Phi^{\wedge}( \beta)$ for every ${\nu \in \R^d}$ and  $\beta \in \R^{k_d}$. This inequality is preserved for the supremum over $\beta \in \R^{k_d}$ and we obtain $\Phi(\nu) \geq \sup\limits_{\beta \in \R^k_d} \langle\beta, \m(\nu)\rangle - \Phi^{\wedge}( \beta) = \Phi^{\wedge \vee}(\nu)$.

				(iii): This statement is a direct consequence of the definition of the polyconvex conjugation.

				(iv): Part (ii) shows $\Phi \geq \Phi^{\wedge \vee}$, thus it is sufficient to show that $\Phi \leq \Phi^{\wedge \vee}$.
				Let $g$ be a polyconvex lower semicontinuous representative of $\Phi$. We note that 
				\begin{equation*}
					g^*(\beta) = \sup\limits_{x \in \R^{k_d}} \langle \beta , x \rangle - g(x)  \geq \sup\limits_{\nu \in \R^{d}} \langle \beta , \m(\nu) \rangle - g(\m(\nu)) = \Phi^{\wedge}(\beta)\,.
				\end{equation*}
				Applying the convex conjugation on this inequality yields 
				$g^{**} \leq \Phi^{\wedge *}$ (where the convex conjugation has to be formally extended to functions that are identical to $\infty$).
				Since $g$ is lower semicontinuous and convex, we get with Lemma \ref{lem:LF-Conj} that $g= g^{**}$ and can deduce for arbitrary $\nu \in \R^d$:
				\begin{equation*}
					\Phi(\nu)= g(\m(\nu)) = g^{**}(\m(\nu)) \leq \Phi^{\wedge *}(\m(\nu)) = \sup\limits_{\beta \in \R^{k_d}} \langle \beta , \m(\nu) \rangle - \Phi^{\wedge}(\beta) = \Phi^{\wedge \vee}(\nu)\,.
				\end{equation*}
			\end{proof}
			
			\section{Polyconvex conjugation for \texorpdfstring{$\Pi(d)$}{Pid}-invariant functions}\label{sec:Duality}
			
			In this section, we show that for $\Pi(d)$-invariant functions the dual polyconvex conjugation does not only provide a lower semicontinuous polyconvex function as deduced in Proposition \ref{prop:polyconvexConjugation}, but that the resulting function is also lower semicontinuous signed singular value polyconvex, i.e.~the corresponding function $W\colon \R^{d \times d}\to \R_\infty$ is lower semicontinuous polyconvex.
			
			\begin{proposition}\label{prop:PolyConjugationPid}
				Let $d \in \{2,3\}$ and  $\Phi \colon \R^d \to \R_\infty$ be $\Pi(d)$-invariant. Then, $\Phi^{\wedge \vee}$ is lower semicontinuous signed singular value polyconvex.
			\end{proposition}
			
			In the first half of this section, we show that the elementary function $\nu \mapsto \Phi_\beta(\nu) \coloneqq \max\limits_{S \in \Pd} \langle \beta, \m(S\nu) \rangle$ is lower semicontinuous signed singular value polyconvex, i.e.~induces a lower semicontinuous polyconvex function on $\R^{d \times d}$.
			In the second half, we use this to show Proposition \ref{prop:PolyConjugationPid} and, then, to deduce Theorem \ref{thm:thm-for-diagonal}.
			
			\smallskip
			
			\textbf{Elementary singed singular value polyconvex function}
			\begin{proposition}[Elementary singed singular value polyconvex functions]\label{prop:Lambda_lsc-svpc}
				Let $d \in \{2,3\}$. For every $\beta \in \R^{k_d}$, the mapping
				\begin{equation*}
					\Phi_\beta \colon \R^{d} \to \R\,, \quad \nu \mapsto \Phi_\beta(\nu) \coloneqq \max\limits_{S \in \Pd} \langle \beta, \m(S\nu) \rangle\,,
				\end{equation*}
				is $\Pi(d)$-invariant and lower semicontinuous signed singular value polyconvex.
			\end{proposition} 
			
			The proof of Proposition \ref{prop:Lambda_lsc-svpc} is based on Lemma \ref{lem:B=D}, which extends \cite[Proposition 3.5]{Mie05} to signed singular values.
			In order to convince ourselves that this extension is valid and to give deeper insights, we recap intermediate results (see Lemma \ref{lem:Schur} and \ref{lem:B=D}) and adapt the proof of \cite[Proposition 3.5]{Mie05}.
			
			Let $\P \colon \R^{K_d} \to \R^{k_d}$ denote the projection of $\R^{K_d}$ onto the non-trivial entries of $\M \circ \diag$ so that $\P \M(\diag(\nu)) = \m(\nu)$ in particular.
			Let $\odot \colon \R^{d \times d} \times\R^{d \times d} \to \R^{d \times d}$ be the Schur product for matrices defined by the pointwise multiplication, i.e.~$A \odot B = (A_{ij} B_{ij})_{i,j =1, \dots, d}$ for $A,B \in \R^{d \times d}$. We write $A:B = \sum\limits_{i,j =1, \dots, d} A_{ij} B_{ij}$  for the Frobenius product of matrices $A,B \in \R^{d \times d}$. We obtain analogously to \cite[Lemma 3.3]{Mie05} the following result. 
			\begin{lemma}\label{lem:Schur}
				Let $d \in \{2,3\}$, $\beta \in \R^{k_d}$, $\nu \in \R^d$ and $R_1, R_2 \in \SO(d)$. Then, it holds
				\begin{equation}\label{eq:LinearBySchur}
					\langle \beta, \P( \M(R_1\diag(\nu) R_2) )\rangle = (R_1 \odot R_2^\top) \colon N + \beta_{k_d} \prod_{i=1}^d \nu_i
				\end{equation}
				with
				\begin{equation}\label{eq:def:N}
					N = \begin{cases}
						(\beta_1, \beta_2)^\top \otimes \nu &\textrm{ for } d= 2,
						\\
						(\beta_1, \beta_2, \beta_3)^\top \otimes \nu + (\beta_4, \beta_5, \beta_6)^\top \otimes \tilde{\nu} &\textrm{ for } d= 3
					\end{cases}
				\end{equation}
				and $\tilde{\nu} = (\nu_2 \nu_3,\,\nu_1 \nu_3,\, \nu_1 \nu_2)^\top$.
			\end{lemma}
			
			In Lemma \ref{lem:B=D}, we consider the supremum of $A \mapsto A  : N + c \prod_{i=1}^d\nu_i$ over
			\begin{equation*}
				\mathcal{T}_d \coloneqq \{A = R_1 \odot R_2 \mid R_1, R_2 \in \SO(d) \}.
			\end{equation*}
			A linear function over a compact set $T$ always attains its maximum on $\operatorname{ex}(\operatorname{conv}(T))$, the extremal points of the convex hull of $T$. For $C,T \subset \R^{d \times d}$ with $C$ convex, we define
			\begin{align*}
				&\operatorname{conv}(T) \coloneqq \left\{\sum_{j=1}^{d^2+1} \lambda_j A_j\mid \lambda_j \geq 0,\, \sum_{j=1}^{d^2+1} \lambda_j =1,\, A_j \in T \right\},\\
				&\operatorname{ex}(C) \coloneqq \{ A \in C \mid C\setminus \{A\} \textrm{ is convex} \}.
			\end{align*}
			For $d \in \{2,3\}$, the set $\operatorname{ex}(\operatorname{conv}(\mathcal{T}_d))$ was explicitly computed and identified as $\Pd$ in \cite[Proposition 3.4]{Mie05}.
			Together with Lemma \ref{lem:Schur}, we can conclude the following polynomial inequality of von Neumann type analogously to \cite{Mie05}.
			\begin{lemma}[A polynomial inequality of von Neumann type]\label{lem:B=D}
				Let $d \in \{2,3\}$, $\beta\in \R^{k_d}$ and $\nu \in \R^d$. Then,
				\begin{equation*}
					\sup\limits_{R_1, R_2 \in \SO(d)} \langle\beta, \P \left(\M( R_1 \diag(\nu) R_2) \right) \rangle
					=
					\max\limits_{S \in \Pd} \langle\beta, \m(\nu)\rangle. 
				\end{equation*}
			\end{lemma}
			\begin{proof}
				By the Krein--Milman theorem, linear functionals attain their extrema over a compact and convex set on extremal points. Thus, we obtain with Lemma \ref{lem:Schur}:
				\begin{align*}
					\sup\limits_{R_1, R_2 \in \SO(d)} \langle \beta, \P \left(\M( R_1 \diag(\nu) R_2) \right) \rangle
					= 
					\sup\limits_{R_1, R_2 \in \SO(d)} R_1 \odot R_2^\top \colon N + \beta_{k_d} \prod_{i=1}^d\nu_i
					\\
					= \sup\limits_{A \in \mathcal{T}_d} A \colon N + \beta_{k_d} \prod_{i=1}^d\nu_i
					=
					\max\limits_{S \in \Pd} S :N + \beta_{k_d} \prod_{i=1}^d\nu_i = \max\limits_{S \in \Pd} \langle\beta, \m(S\nu)\rangle\,,
				\end{align*}
				where $N$ is given by \eqref{eq:def:N}. The last step follows for $d=2$ by 
				\begin{align*}
					S \colon N = S \colon ((\beta_1, \beta_2)^\top \otimes \nu) = \langle (\beta_1, \beta_2)^\top, S \nu\rangle
				\end{align*}
				and for $d =3$ with
				\begin{align*}
					S \colon N = S \colon ((\beta_1, \beta_2, \beta_3)^\top \otimes \nu+ ((\beta_4, \beta_5, \beta_6)^\top \otimes \tilde{\nu}))
					\\
					= \langle(\beta_1, \beta_2, \beta_3)^\top, S\nu\rangle+ \langle(\beta_4, \beta_5, \beta_6)^\top, \widetilde{S\nu} \rangle.
				\end{align*} and  \eqref{eq:Snu=Snu}
				where $\widetilde{S\nu}= \left(\begin{array}{c}
					(S\nu)_2 (S\nu)_3\\ (S\nu)_1 (S\nu)_3\\
					(S\nu)_1 (S\nu)_2
				\end{array}\right)$.
			\end{proof}
			
			\begin{proof}[Proof of Proposition \ref{prop:Lambda_lsc-svpc}]
				We show the result for $d =3$, the case $d=2$ follows similarly.
				We note that for every $\beta \in \R^{k_d}$ and $R_1, R_2 \in \SO(3)$, the mapping
				\begin{equation*}
					(A,B,c) \mapsto \langle \beta, \operatorname{P}(R_1 A R_2,\,R_1 B R_2,\, c) \rangle
				\end{equation*}
				is linear, thus convex and continuous.
				Since convexity and lower semicontinuity are preserved for the supremum, the mapping
				\begin{equation*}
					G_\beta(A,B,c) \coloneqq \sup\limits_{R_1, R_2 \in \SO(3)} \langle\beta, \operatorname{P} (R_1 A R_2,\,R_1 B R_2,\,c) \rangle
				\end{equation*}
				is convex and lower semicontinuous. Thus, $W_\beta$ defined by 
				\begin{equation*}
					W_{\beta}(F) \coloneqq G_{\beta}(F, \,\adj(F)^\top,\, \det(F)) = G_{\beta}(\M(F))
				\end{equation*} is lower semicontinuous polyconvex. We observe that $W_{\beta}$ is isotropic since for every $\tilde{R}_1, \tilde{R}_2 \in \SO(3)$:
				\begin{align*}
					W_\beta(F) &= \sup\limits_{R_1, R_2 \in \SO(3)} \langle\beta, \operatorname{P} (R_1 F R_2,\,R_1 \adj(F)^\top R_2,\, \det(F)) \rangle
					\\
					&=
					\sup\limits_{R_1, R_2 \in \SO(3)} \langle\beta, \operatorname{P} (R_1 \tilde{R_1}F \tilde{R_2}R_2,\, R_1\tilde{R_1} \adj(F)^\top\tilde{R_2}R_2,\, \det(F)) \rangle
					\\
					&=\sup\limits_{R_1, R_2 \in \SO(3)} \langle\beta, \operatorname{P} (R_1 \tilde{R_1}F \tilde{R_2}R_2,\, R_1 \adj(\tilde{R_1}F\tilde{R_2})^\top R_2,\, \det(F)) \rangle
					\\
					&=W_\beta(\tilde{R_1}F\tilde{R_2})\,.
				\end{align*}
				With Lemma \ref{lem:B=D}, we can reduce the supremum over $R_1,R_2 \in \SO(3)$ to the following maximum
				\begin{align*}
					W_\beta(\diag(\nu))
					&= \sup\limits_{R_1, R_2 \in \SO(3)} \langle\beta, \operatorname{P} (R_1 \diag(\nu) R_2,\, R_1 \adj( \diag(\nu))^\top R_2,\, \det(\diag(\nu)) \rangle
					\\
					&=
					\max\limits_{S \in \Pd} \langle\beta, \m(S\nu)\rangle .
				\end{align*}
				Hence, $\nu \mapsto \max\limits_{S \in \Pd} \langle\beta, \m(S\nu)\rangle $ is  lower semicontinuous signed singular value polyconvex.
			\end{proof}

			\begin{proof}[Proof of Proposition \ref{prop:PolyConjugationPid}]
				We show the result for $d=3$, where $d=2$ can be treated analogously.
				For $S \in \Pi(3)$ and $\beta \in \R^{k_3}$, we define
				\begin{equation*}
					S^\top_{\text{P}}  \beta \coloneqq (S^\top(\beta_1, \beta_2, \beta_3)^\top, \, S^\top(\beta_4, \beta_5, \beta_6)^\top, \, \beta_7)\,
				\end{equation*}
				and note that $S^\top_{\text{P}} \colon \R^{k_3}\to \R^{k_3}$ is bijective since $S$ is bijective on $\R^{3}$.
				For arbitrary $\nu \in \R^3$ and $\tilde{\nu} =(\nu_2 \nu_3,\, \nu_1\nu_3,\, \nu_1\nu_2)^\top$, we obtain with \eqref{eq:Snu=Snu}:
				\begin{align*}
					\langle S^\top_{\text{P}} \beta , \m (\nu) \rangle 
					&= \langle S^\top(\beta_1, \beta_2, \beta_3)^\top,  \nu \rangle
					+
					\langle S^\top(\beta_4, \beta_5, \beta_6)^\top,  \tilde{\nu} \rangle
					+
					\beta_7 \nu_1\nu_2 \nu_3 
					\\
					&= \langle (\beta_1, \beta_2, \beta_3)^\top,  S\nu \rangle
					+
					\langle (\beta_4, \beta_5, \beta_6)^\top,  S\tilde{\nu} \rangle
					+
					\beta_7 \nu_1\nu_2 \nu_3 
					=\langle \beta , \m (S \nu) \rangle \,.
				\end{align*}
				With the $\Pi(3)$-invariance of $\Phi$, we calculate for arbitrary $S\in \Pi(3)$ and $\beta \in \R^{k_3}$ using that $S \colon \R^3\to \R^3$ is bijective:
				\begin{align*}
					\Phi^{\wedge}(S_{\text{P}}^\top \beta) = \max\limits_{\nu \in \R^3}\langle S_{\text{P}}^\top \beta , \m(\nu)\rangle - \Phi(\nu) = \max\limits_{\nu \in \R^3}\langle \beta , \m(S\nu)\rangle - \Phi(S \nu) = \Phi^{\wedge}(\beta) \,.
				\end{align*}
				We deduce further 
				\begin{align*}
					\Phi^{\wedge \vee}(\nu) &\coloneqq \sup\limits_{\beta\in \R^{k_d}} \langle\beta, \m(\nu)\rangle - \Phi^{\wedge}(\beta) = \sup\limits_{\beta\in \R^{k_d}} \max\limits_{S\in \Pi(3)}\langle S_{\text{P}}^\top\beta , \m(\nu)\rangle - \Phi^{\wedge}(S_{\text{P}}^\top \beta)
					\\
					&=
					\sup\limits_{\beta\in \R^{k_d}} \max\limits_{S\in \Pi(3)}\langle \beta , \m(S\nu)\rangle - \Phi^{\wedge}(\beta) \,.
				\end{align*}
				By Proposition \ref{prop:Lambda_lsc-svpc}, for every $\beta \in \R^{k_d}$, the map $\nu \mapsto \max\limits_{S\in \Pi(d)} \langle\beta, \m(S\nu)\rangle$ is lower semicontinuous signed singular value polyconvex. Subtracting the constant $\Phi^{\wedge}(\beta)$  does not affect this property.
				The lower semicontinuity as well as convexity are preserved for the supremum over $\beta \in \R^{k_d}$ and, thus, $\Phi^{\wedge \vee}$ is lower semicontinuous signed singular value polyconvex.
				
			\end{proof}

			With Proposition \ref{prop:polyconvexConjugation}, we can show that $\Phi^{\wedge \vee}$ provides the lower semicontinous signed singular value polyconvex envelope of $\Phi$.
			\begin{corollary}\label{cor:main}
				Let $ d \in \{2,3\}$ and $\Phi \colon \R^d \mapsto \R_\infty$ be isotropic. Then,
				\begin{align*}
					\Phi^{\wedge \vee}(\nu) = \sup\{\Psi(\nu) \mid \Psi \leq \Phi, \,\Psi &\textrm{ is  lower semicontinuous }\\ &\text{ signed singular value polyconvex} \}\,
				\end{align*}
				and, thus, $\Phi^{\wedge \vee}(\nu)\leq \SvPc\Phi \leq \Phi$ for
				\begin{equation*}
					\SvPc \Phi(\nu) \coloneqq \sup \left\{\Psi(\nu) \mid \Psi \leq \Phi,\,\Psi \textrm{ is signed singular value polyconvex} \right\}\,,
				\end{equation*}
				where we implicitly assume that the functions $\Psi$ are $\Pi(d)$-invariant since we have defined signed singular value polyconvexity only for $\Pi(d)$-invariant functions.
			\end{corollary}
			\begin{proof}We show Corollary \ref{cor:main} by contradiction.
				Assume there exists $\Psi$, which is  lower semicontinuous signed singular value polyconvex and $\nu_0 \in \R^d$ such that $\Psi \leq \Phi$ and $\Phi^{\wedge \vee}(\nu_0) < \Psi(\nu_0)$.
				The pointwise maximum $\max\{\Psi, \Phi^{\wedge \vee} \}$ is  lower semicontinuous signed singular value polyconvex and since $\Phi^{\wedge \vee} \leq \Phi$ (cf.~Proposition \ref{prop:polyconvexConjugation}(i)), we can assume without loss of generality that $\Phi^{\wedge \vee} \leq \Psi \leq \Phi$.
				With Proposition \ref{prop:polyconvexConjugation}(iv) and (iii) we deduce $\Psi = \Psi^{\wedge \vee}\leq \Phi^{\wedge \vee}$, which is a contradiction to $\Phi^{\wedge \vee}(\nu_0) < \Psi(\nu_0)$.
				
				The inequality $\Phi^{\wedge \vee}(\nu)\leq \SvPc \Phi \leq \Phi$ follows from the first part of Corollary~\ref{cor:main}.
			\end{proof}
			
			\begin{proof}[Proof of Theorem \ref{thm:thm-for-diagonal}]
				We have already observed that the desired result is equivalent to the statement:
				A $\Pi(d)$-invariant function $\Phi \colon \R^d \to \R_\infty$ is lower semicontinuous singed singular value polyconvex if and only if it is is lower semicontinuous polyconvex. The trivial implication is shown in Lemma \ref{lem:EquivPolyconvexityPhi}.
				For the reverse direction, let $\Phi$ be lower semicontinuous polyconvex, i.e.~there exists a lower semicontinuous convex function $g$ such that $g\circ \m = \Phi$ holds.
				We choose $h$ as in \eqref{eq:def:h} and thus, $g \leq h$. We deduce
				\begin{equation*}
					\Phi^{\wedge\vee}(\nu) \overset{\text{Lemma \ref{lem:Equivalence:Conj}}}{=} h^{**}(\m(\nu)) \overset{\text{Lemma \ref{lem:LF-Conj}(iii)}}{\geq} g^{**}(\m(\nu)) \overset{\text{Lemma \ref{lem:LF-Conj}(iv)}}{=} g(\m(\nu)) = \Phi(\nu) \,.
				\end{equation*}
				Proposition \ref{prop:polyconvexConjugation}(ii) gives the reverse inequality $\Phi^{\wedge\vee} \leq \Phi$ and, thus, $\Phi^{\wedge\vee} = \Phi$. 
				Proposition \ref{prop:polyconvexConjugation}(i) shows that $\Phi^{\wedge\vee}$ and, thus, $\Phi$ is lower semicontinuous signed singular value polyconvex.
			\end{proof}
						 
			\section{Some remarks on polyconvexity with respect to matrix invariants}\label{sec:Invariants}
			Isotropic functions $W$ with $W(F)= \infty$ for $\det(F)< 0$ can be formulated by means of singular values as in \eqref{eq:IsoSingularValues}.
			The singular values of $F$ are given by the zeros of the characteristic polynomial of $U=\sqrt{F^\top F}$.
			Since the coefficients of this polynomial are given by the matrix invariants $I(U)$ one can express such a function $W$ by means of the matrix invariants, i.e.~there exists $\psi \colon [0,\infty)^d \to \R_\infty$ such that $W(F)= \psi(I(U))$ for $\det(F)> 0$ and $W(F) = \infty$ else. For $d=3$, the invariants of $U$ are ${I(U) = (\operatorname{tr}(U),\, \operatorname{tr}(\adj(U)),\, \det(U))}$.
			
			For $d\in \{2,3\}$ and $\psi\colon [0,\infty)^d \to \R$ convex and non-decreasing in all but the last argument and $W$ given as above, it was shown in \cite{Ste03} using the polyconvexity criterion of Ball \cite[Theorem 5.2]{Bal77} that $W$ is polyconvex. While, this particular choice of $g$ weakens the criterion, it can simplify its applications and allows the direct consideration of functions naturally formulated in terms of matrix invariants. 
			We adapt the approach of \cite{Ste03} and use  Theorem~\ref{thm:thm-for-diagonal} instead of \cite[Theorem 5.2]{Bal77}. This leads to a new sufficient criterion for polyconvexity for functions which are given in terms of the elementary symmetric polynomials of the signed singular values. 
			This new criterion replaces the restriction that $\psi$ is non-decreasing by the symmetry assumption \eqref{eq:Sym:psi}. 
			
			We denote the elementary symmetric polynomials by $e=e(x)$, i.e.~
			\begin{align}\label{eq:ElemtarySymmetricPoly}
				\begin{aligned}
					&e\colon \R^2 \to \R^2, \ 
					x \mapsto (x_1 + x_2, \, x_1x_2)&& \textrm{ for } d=2\,,
					\\
					&e\colon \R^3 \to \R^3, \ 
					x \mapsto (x_1 + x_2 + x_3, \, x_2 x_3 +x_1 x_3 +x_1x_2, \, x_1 x_2 x_3) && \textrm{ for } d=3\,.
				\end{aligned}
			\end{align}
			
			\begin{lemma}[Injectivity of the elementary symmetric polynomials]\label{lem:Inj:ElemSymPol}
				Let $e \colon \R^d \to \R^d$ be defined by \eqref{eq:ElemtarySymmetricPoly}. Then, $e$ is injective up to permutations in $\S(d)$, i.e.~$e(x) = e(\tilde{x})$ if and only if $\tilde{x} = S x$ for $S \in \S(d)$.
			\end{lemma}
			\begin{proof}
				Let $x, \tilde{x} \in \R^d$ with $e(x) = e(\tilde{x}) \in \R^d$. 
				We identify the following polynomials 
				\begin{align*}
					\prod_{i=1}^d (\lambda -x_i) &= \lambda^d - \sum\limits_{i=1}^d \lambda^{d-i} (-1)^{i} e_i(x) = \lambda^d - \sum\limits_{i=1}^d \lambda^{d-i} (-1)^{i} e_i(\tilde{x}) =\prod_{i=1}^d (\lambda - \tilde{x}_i)
				\end{align*}
				and the fundamental theorem of algebra implies $\tilde{x} = S \tilde{x}$ for some $S \in \S(d) $.
			\end{proof}
			Due to the injectivity of $e$ given in Lemma \ref{lem:Inj:ElemSymPol}, for every $\S(d)$-invariant function $\Phi \colon \R^d \to \R_\infty$, i.e.~$\Phi(\nu ) = \Phi(S\nu)$ for all $\nu \in \R^d$ and $S \in S(d)$, there exists a function ${\psi \colon \R^d \to \R_\infty}$ such that
			\begin{equation}\label{eq:Phi=psi}
				\Phi(\nu) = \psi(e(\nu))\,.
			\end{equation}
			The function $\psi$ in \eqref{eq:Phi=psi} is uniquely defined on the image of $e$. The other way around, every function ${\psi \colon \R^d \to \R_\infty}$ defines an $\S(d)$-invariant function $\Phi \colon \R^d \to \R_\infty$ via \eqref{eq:Phi=psi}. However, in our application of isotropic functions, $\Phi$ is not only $\S(d)$- but also $\Pd$-invariant. Thus, $\psi$ has to satisfy the symmetry
			\begin{equation}\label{eq:Sym:psi}
				\psi( e(\nu)) = \psi(e(\diag(\varepsilon) \nu)) \quad \text{for all }\nu \in \R^d \text{ and } \epsilon \in \{-1,1\}^d \text{ with } \prod_{i=1}^d \epsilon_i = 1
			\end{equation}
			in order to generate a $\Pi(d)$-invariant function $\Phi$ via \eqref{eq:Phi=psi}.
			
			With this at hand, we can formulate a sufficient condition for polyconvexity in terms of $\psi$:
			\begin{proposition}\label{prop:Steigman:criterion}
				Let $d \in \{2,3\}$ and assume that $\psi \colon \R^d \to \R_\infty$ is convex, lower semicontinuous and satisfies \eqref{eq:Sym:psi}. Then, $\Phi\colon \R^d \to \R_\infty$ given by \eqref{eq:Phi=psi} is  lower semicontinuous signed singular value polyconvex.
			\end{proposition}
			\begin{proof}
				We prove the result for $d =3$; it follows in the same way for $d=2$.
				
				We choose $g(\nu_1,\nu_2,\nu_3, \mu_1, \mu_2, \mu_3, \delta) \coloneqq \psi (\nu_1 + \nu_2 +\nu_3, \mu_1 +\mu_2+\mu_3,\delta)$.
				It can be easily seen that $g \colon \R^{k_3} \to R_\infty$ is convex if and only if $\psi$ is convex. Moreover, $g$ is lower semicontinuous if and only if $\psi$ is lower semicontinuous. 
				Theorem \ref{thm:Phi=g} implies that $\Phi$ is  lower semicontinuous signed singular value polyconvex.
			\end{proof}

			Proposition \ref{prop:Steigman:criterion} gives only a sufficient but not a necessary condition for the  lower semicontinuous signed singular value polyconvexity of $\Phi$. The proof uses a specific choice of the representative $g$ based on $\psi$ but does not exclude other choices for $g$.
			
			We note that for $F\in \R^{d \times d}$ with $\det(F) >0$, we can choose the signed singular values positive, i.e.~$\nu \in (0,\infty)^d$ and have $e(\nu) =  I(U) = (\operatorname{tr}(U),\, \operatorname{tr}(\adj(U)),\, \det(U))$, where $U$ is the symmetric positive definite right stretch tensor of $F$, i.e.~$U^2 = F^\top F$.
			As a consequence of Proposition \ref{prop:Steigman:criterion}, we can replace the non-decreasing property in the polyconvexity criterion of \cite{Ste03} by the symmetry assumption \eqref{eq:Sym:psi}:
			\begin{corollary}
				Let $d \in \{2,3\}$ and $\psi \colon \R^{d} \to \R_\infty$ be a convex and lower semicontinuous function satisfying \eqref{eq:Sym:psi} and $\psi(\alpha_1, \dots, \alpha_d) = \infty$ for $\alpha_d \leq 0$. Then, the function $W \colon \R^{d \times d} \to \R_\infty$ be given by
				\begin{equation*}
					W(F) = \begin{cases}
						\psi(I(U)) & \text{ if }\det(F) > 0\,,
						\\
						\infty & \text{ if } \det(F) \leq 0\,,
					\end{cases}
				\end{equation*}
				where $U$ is the symmetric positive definite right stretch tensor of $F$, i.e.~$U^2 = F^\top F$, and $I(U)$ denotes its invariants, ${I(U) = (\operatorname{tr}(U),\, \operatorname{tr}(\adj(U)),\, \det(U))}$.
				Then, $W$ is isotropic and lower semicontinuous polyconvex.
			\end{corollary}

			\section{Conclusion}
			
			We identified isotropic functions ($W$) on $\R^{d \times d}$ by $\Pi(d)$-invariant functions ($\Phi$) on $\R^d$. For functions on $\R^d$, we use a notion of polyconvexity, which arises naturally when replacing the minors of matrices by minors of vectors. We introduced the notion of singed singular value polyconvexity for $\Pi(d)$-invariant functions via the polyconvexity of the corresponding isotropic function. Using the polyconvex conjugation, an extension of the Fenchel--Legendre transformation, we showed for $\Pi(d)$-invariant functions that they are polyconvex if and only if they are signed singular value polyconvex. This leads to our main result:
			isotropic functions are polyconvex if and only if their restriction to the set of diagonal matrices is polyconvex. 
			
			It provides  the following dimension reduction: Instead of seeking a convex representatives  $G \colon \R^{k_d} \to \R_\infty$, it suffices to seek for a convex representatives $g \colon \R^{k_d} \to \R_\infty$, with $k_2 = 3$ instead of $K_2= 5$ and $k_3= 7$ instead of $K_3 = 19$.
						
			Furthermore, we derived a new criterion for the polyconvexity of isotropic functions which are given in terms to the matrix invariants of the right stretch tensor $\sqrt{F^\top F}$ instead of the deformation gradient $F$ or, more generally, the elementary symmetric polynomials of the singular values of $F$.
			
			\section*{Acknowledgement}
			D.~Wiedemann would like to thank the Marianne-Plehn-Programm for funding.
			Furthermore, both authors acknowledge the Deutsche Forschungsgemeinschaft (DFG) for funding within the Priority Program 2256 (“Variational Methods for Predicting Complex Phenomena in Engineering Structures and Materials”), Project ID 441154176, reference ID PE1464/7-1.
			
			We would like to thank A.~Mielke and D.~Balzani for fruitful discussions and T.~Neumeier und D.~Peterseim for helpful remarks on the manuscript.
			
			\subsection*{Statements and Declarations}
			The authors have no financial or non-financial interests that are directly or indirectly
			related to the work submitted for publication.
			Data sharing not applicable to this article.

						\bibliographystyle{alpha}
			\bibliography{IsotropicPolyconvexity}

		\end{document}